\documentclass{article}
\usepackage{amsmath, amsfonts, amssymb, mathrsfs, amsthm,accents, sectsty,hyphenat, palatino}
\usepackage[all]{xy}
\usepackage[usenames]{color}
\def\mf#1{\mathfrak{#1}}
\def\mc#1{\mathcal{#1}}

\def\tx#1{{\rm #1}}
\def\tb#1{\textbf{#1}}

\def\tr{\tx{tr}\,}
\def\R{\mathbb{R}}

\def\C{\mathbb{C}}
\def\Q{\mathbb{Q}}

\def\ol#1{\overline{#1}}

\def\hat{\widehat}

\def\rw{\rightarrow}

\def\lrw{\longrightarrow}

\def\srw{\twoheadrightarrow}

\def\<{\langle}
\def\>{\rangle}

\newenvironment{mytitle}
{\begin{center}\large\sc}
{\end{center}}

\newtheorem{thm}{Theorem}[section]
\newtheorem{lem}[thm]{Lemma}
\newtheorem{pro}[thm]{Proposition}

\newtheorem{cor}[thm]{Corollary}

\newtheorem{fct}[thm]{Fact}

\sectionfont{\center\sc\normalsize}
\subsectionfont{\bf\normalsize}

\numberwithin{equation}{section}
\renewcommand{\-}{\hyp{}}

\def\phi{\varphi}

\newlength{\cplxcorr}

\begin{document}

\begin{mytitle} Genericity and contragredience in the local Langlands correspondence \end{mytitle}
\begin{center} Tasho Kaletha \end{center}
\begin{abstract}
We prove the recent conjectures of Adams-Vogan and D. Prasad on the behavior of the local Langlands correspondence with respect to taking the contragredient of a representation. The proof holds for tempered representations of quasi-split real $K$-groups and quasi-split p-adic classical groups (in the sense of Arthur). We also prove a formula for the behavior of the local Langlands correspondence for these groups with respect to changes of the Whittaker data.
\end{abstract}

\section{Introduction}

The local Langlands correspondence is a conjectural relationship between certain representations of the Weil or Weil-Deligne group of a local field $F$ and finite sets, or packets, of representations of a locally compact group arising as the $F$-points of a connected reductive algebraic group defined over $F$. In characteristic zero, this correspondence is known for $F=\R$ and $F=\C$ by work of Langlands \cite{Lan89} and was later generalized and reinterpreted geometrically by Adams, Barbasch, and Vogan \cite{ABV92}. Furthermore, many cases are known when $F$ is a finite extension of the field $\Q_p$ of $p$-adic numbers. Most notably, the correspondence over $p$-adic fields is known when the reductive group is $\tx{GL}_n$ by work of Harris-Taylor \cite{HT01} and Henniart \cite{He00}, and has very recently been obtained for quasi-split symplectic and orthogonal groups by Arthur \cite{Ar11}. Other cases include the group $U_3$ by work of Rogawski, $\tx{Sp}_4$ and $\tx{GSp}_4$ by work of Gan-Takeda. For general connected reductive groups, there are constructions of the correspondence for specific classes of parameters, including the classical case of unramified representations, the case of representations with Iwahori-fixed vector by work of Kazhdan-Lusztig, unipotent representations by work of Lusztig, and more recently regular depth-zero supercuspidal representations by DeBacker-Reeder \cite{DR09}, very cuspidal representations by Reeder \cite{Re08}, and epipelagic representations by the author \cite{Ka12}.

The purpose of this paper is to explore how the tempered local Langlands correspondence behaves with respects to two basic operations on the group. The first operation is that of taking the contragredient of a representation. In a recent paper, Adams and Vogan \cite{AV12} studied this question for the general (not just tempered) local Langlands correspondence for real groups. They provide a conjecture on the level of $L$-packets for any connected reductive group over a local field $F$ and prove this conjecture when $F$ is the field of real numbers. One of our main results is the fact that this conjecture holds for the tempered $L$-packets of symplectic and special orthogonal $p$-adic groups constructed by Arthur. In fact, inspired by the work of Adams and Vogan, we provide a refinement of their conjecture to the level of representations, rather than packets, for the tempered local Langlands correspondence. We prove this refinement when $G$ is either a quasi-split connected real reductive group (more generally, quasi-split real $K$-group), a quasi-split symplectic or special orthogonal $p$-adic group, and in the context of the constructions of \cite{DR09} and \cite{Ka12}. In the real case, the results of Adams and Vogan are a central ingredient in our argument. To obtain our results, we exploit the internal structure of real $L$-packets using recent results of Shelstad \cite{Sh08}. In the case of quasi-split $p$-adic symplectic and special orthogonal groups, we prove a result similar to that of Adams and Vogan using Arthur's characterization of the stable characters of $L$-packets on quasi-split $p$-adic classical groups as twisted transfers of characters of $\tx{GL}_n$. After that, the argument is the same as for the real case. The constructions of \cite{DR09} and \cite{Ka12} are inspected directly.

The second basic operation that we explore is that of changing the Whittaker datum. To explain it, we need some notation. Let $F$ be a local field and $G$ a connected reductive group defined over $F$. Let $W'$ be the Weil group of $F$ if $F=\R$ or the Weil-Deligne group of $F$ if $F$ is an extension of $\Q_p$. Then, if $G$ is quasi-split, it is expected that there is a bijective correspondence $(\phi,\rho) \mapsto \pi$. The target of this correspondence is the set of equivalence classes of irreducible admissible tempered representations. The source of this correspondence is the set of pairs $(\phi,\rho)$ where $\phi : W' \rw {^LG}$ is a tempered Langlands parameter, and $\rho$ is an irreducible representation of the finite group $\pi_0(\tx{Cent}(\phi,\hat G)/Z(\hat G)^\Gamma)$. Here $\hat G$ is the complex (connected) Langlands dual group of $G$, and $^LG$ is the $L$-group of $G$. However, it is known that such a correspondence can in general not be unique. In order to hope for a unique correspondence, following Shahidi \cite[\S9]{Sha90} one must choose a Whittaker datum for $G$, which is a $G(F)$-conjugacy class of pairs $(B,\psi)$ where $B$ is a Borel subgroup of $G$ defined over $F$ and $\psi$ is a generic character of the $F$-points of the unipotent radical of $B$. Then it is expected that there exists a correspondence $(\phi,\rho) \mapsto \pi$ as above which has the property that $\pi$ has a $(B,\psi)$-Whittaker functional precisely when $\rho=1$. Let us denote this conjectural correspondence by $\iota_{B,\psi}$. We are interested in how it varies when one varies the Whittaker datum $(B,\psi)$. We remark that there is a further normalization of $\iota_{B,\psi}$ that must be chosen. As described in \cite[\S 4]{KS12}, it is expected that there will be two normalizations of the local Langlands correspondence for reductive groups, reflecting the two possible normalizations of the local Artin reciprocity map.

The reason we study these two questions together is that they appear to be related. Indeed, when one studies how the pair $(\phi,\rho)$ corresponding to a representation $\pi$ changes when one takes the contragredient of $\pi$, one is led to consider $\iota_{B,\psi}$ for different Whittaker data.

We will now go into more detail and describe our expectation for the behavior of the local Langlands correspondence with respect to taking contragredient and changing the Whittaker datum. We emphasize that we claim no originality for these conjectures. Our formula in the description of the contragredient borrows greatly from the paper of Adams and Vogan, as well as from a conversation with Robert Kottwitz, who suggested taking the contragredient of $\rho$. After the paper was written, we were informed by Dipendra Prasad that an equation closely related to \eqref{eq:expwhit} is stated as a conjecture in \cite[\S9]{GGP12}, and that moreover equation \eqref{eq:expcont} is part of a more general framework of conjectures currently being developed by him under the name ``relative local Langlands correspondence''. We refer the reader to the draft \cite{Pr}.

We continue to assume that $F$ is either real or $p$-adic, and $G$ is a quasi-split connected reductive group over $F$. Fix a Whittaker datum $(B,\psi)$. For any Langlands parameter $\phi : W' \rw {^LG}$, let $S_\phi = \tx{Cent}(\phi,\hat G)$. The basic form of the expected tempered local Langlands correspondence is a bijection $\iota_{B,\psi}$ from the set of pairs $(\phi,\rho)$, where $\phi$ is a tempered Langlands parameter and $\rho$ is an irreducible representation of $\pi_0(S_\phi/Z(\hat G)^\Gamma)$ to the set of equivalence classes of irreducible admissible tempered representations. A refinement of this correspondence is obtained when one allows $\rho$ to be an irreducible representation of $\pi_0(S_\phi)$ rather then its quotient $\pi_0(S_\phi/Z(\hat G)^\Gamma)$. The right-hand side is then the set of equivalence classes of tuples $(G',\xi,u,\pi)$, where $\xi : G \rw G'$ is an inner twist, $u \in Z^1(F,G)$ is an element with the property $\xi^{-1}\sigma(\xi)=\tx{Int}(u(\sigma))$ for all $\sigma \in \Gamma$, and $\pi$ is an irreducible admissible tempered representation of $G'(F)$. The triples $(G',\xi,u)$ are called pure inner twists of $G$, and the purpose of this refined version of the correspondence is to include connected reductive groups which are not quasi-split. The idea of using pure inner forms is due to Vogan \cite{Vo93}, and one can find a formulation of this refinement of the correspondence in \cite{Vo93} or \cite[\S3]{DR09}. A further refinement is obtained by allowing $\rho$ to be an irreducible algebraic representation of the complex algebraic group $\bar S_\phi = S_\phi/[S_\phi \cap \hat G_\tx{der}]^\circ$. The right-hand side then is the set of equivalence classes of tuples $(G',\xi,b,\pi)$, where $\xi : G \rw G'$ is an inner twist and $b$ is a basic
element of $Z^1(W,G(\ol{L}))$, where $L$ is the completion of the maximal unramified extension of $F$, and where $b$ gives rise to $\xi$ as in \cite{Ko97}. This further refinement was introduced by Kottwitz in an attempt to include all connected reductive groups into the correspondence (it is known that not every connected reductive group is a pure inner form of a quasi-split group). Indeed, when the center of $G$ is connected, all inner forms of $G$ come from basic elements of $Z^1(W,G(\ol{L}))$. Moreover, one can reduce the general case to that of connected center. An exposition of this formulation of the correspondence can be found in \cite{Ka11}.

We now let $\iota_{B,\psi}$ denote any version of the above conjectural correspondence, normalized so that $\iota_{B,\psi}(\phi,\rho)$ is $(B,\psi)$-generic precisely when $\rho=1$. The set of Whittaker data for $G$ is a torsor for the abelian group $G_\tx{ad}(F)/G(F)$. Using Langlands' construction of a character on $G(F)$ for each element of $H^1(W,Z(\hat G))$, one obtains from each element of $G_\tx{ad}(F)/G(F)$ a character on the finite abelian group $\tx{ker}(H^1(W,Z(\hat G_\tx{sc})) \rw H^1(W,Z(\hat G)))$. This groups accepts a map from $\pi_0(S_\phi/Z(\hat G)^\Gamma)$ for every Langlands parameter $\phi$. In this way, given a pair of Whittaker data $\mf{w}$ and $\mf{w'}$, the element of $G_\tx{ad}(F)/G(F)$ which conjugates $\mf{w}$ to $\mf{w'}$ provides a character on $\pi_0(S_\phi/Z(\hat G)^\Gamma)$, hence also on $\pi_0(S_\phi)$ and $\bar S_\phi$. We denote this character by $(\mf{w},\mf{w'})$. Then we expect that
\begin{equation} \iota_\mf{w'}(\phi,\rho) = \iota_\mf{w}(\phi,\rho\otimes(\mf{w},\mf{w'})^\epsilon),\label{eq:expwhit}\end{equation}
where $\epsilon=-1$ if $\iota_{\mf{w}}$ and $\iota_\mf{w'}$ are compatible with the classical normalization of the local Artin reciprocity map, and $\epsilon=1$ if $\iota_\mf{w}$ and $\iota_\mf{w'}$ are compatible with Deligne's normalization.

To describe how we expect $\iota_{B,\psi}$ to behave with respect to taking contragredients, we follow \cite{AV12} and consider the Chevalley involution on $\hat G$: As is shown in \cite{AV12}, there exists a canonical element of $\tx{Out}(\hat G)$ which consists of all automorphisms of $\hat G$ that act as inversion on some maximal torus. This canonical element provides a canonical $\hat G$-conjugacy class of $L$-automorphisms of $^LG$ as follows. Fix a $\Gamma$-invariant splitting of $\hat G$ and let $\hat C \in \tx{Aut}(\hat G)$ be the unique lift of the canonical element of $\tx{Out}(\hat G)$ which sends the fixed splitting of $\hat G$ to its opposite. Then $\hat C$ commutes with the action of $\Gamma$, and we put ${^LC}$ to be the automorphism of $\hat G \times W$ given by $\hat C \times \tx{id}$. If we change the splitting of $\hat G$, there exists \cite[Cor. 1.7]{Ko84} an element $g \in \hat G^\Gamma$ which conjugates it to the old splitting. This element also conjugates the two versions of $\hat C$, and hence also the two versions of ${^LC}$. We conclude that that $\hat G$-conjugacy class of ${^LC}$ is indeed canonical. Thus, for any Langlands parameter $\phi : W' \rw {^LG}$, we have a well-defined (up to equivalence) Langlands parameter ${^LC}\circ\phi$. The automorphism $\hat C$ restricts to an isomorphism $S_\phi \rw S_{{^LC}\circ\phi}$ and for each representation $\rho$ of $\bar S_\phi$ we can consider the representation $\rho\circ\hat C^{-1}$ of $\bar S_{^LC\circ\phi}$. When $\phi$ is tempered, we expect
\begin{equation} \iota_{B,\psi}(\phi,\rho)^\vee = \iota_{B,\psi^{-1}}({^LC}\circ\phi,\rho^\vee\circ\hat C^{-1}). \label{eq:expcont}\end{equation}
For this formula it is not important whether $\iota_{B,\psi}$ is normalized with respect to the classical or Deligne's normalization of the local Artin map, as long as $\iota_{B,\psi^{-1}}$ is normalized in the same way.

We will now briefly describe the contents of this paper. In Section \ref{sec:recall}, we recall the fundamental results of Arthur and Shelstad on the endoscopic classification of tempered representations of real and classical $p$-adic groups. In Section \ref{sec:chanwhit} we will describe more precisely the construction of the character $(\mf{w},\mf{w'})$ alluded to in this introduction, and will then prove Equation \eqref{eq:expwhit}. Section \ref{sec:cont} is devoted to the proof of Equation \eqref{eq:expcont} for tempered representations of quasi-split real $K$-groups and quasi-split symplectic and special orthogonal $p$-adic groups. Finally, in Section \ref{sec:explicit} we consider general $p$-adic groups and parameters of zero or minimal positive depth, and prove \eqref{eq:expcont} for those cases as well.

\tb{Acknowledgements:} The author would like to thank Jeff Adams for discussing with him the paper \cite{AV12}, Diana Shelstad for enlightening discussions concerning \cite{KS12} and the normalizations of transfer factors, Robert Kottwitz and Peter Sarnak for their reading of an earlier draft of this paper, Sandeep Varma for pointing out an inaccuracy in an earlier draft, and Dipendra Prasad for pointing out that equations very similar to \eqref{eq:expwhit} and \eqref{eq:expcont} appear as conjectures in \cite{GGP12} and \cite{Pr}.

\tableofcontents

\section{Results of Arthur and Shelstad} \label{sec:recall}
In this section we will recall the results of Arthur and Shelstad on the inversion of endoscopic transfer, which will be an essential ingredient in our proofs. The formulation in the real case is slightly more complicated due to the fact that semi-simple simply-connected real groups can have non-trivial Galois-cohomology, so we will describe the $p$-adic case first.

Let $F$ be a $p$-adic field. Arthur's results apply to groups $G$ which are either the symplectic group, or the split special odd orthogonal group, or the split or quasi-split special even orthogonal groups. The case of even orthogonal groups is slightly more subtle, so let us first assume that $G$ is either symplectic or odd orthogonal. Fix a Whittaker datum $(B,\psi)$. Let $\phi : W' \rw {^LG}$ be a tempered Langlands parameter and put $C_\phi = \pi_0(\tx{Cent}(\phi,\hat G)/Z(\hat G)^\Gamma)$. Arthur's recent results \cite[\S2]{Ar11} imply that there exists an $L$-packet $\Pi_\phi$ of representations of $G(F)$ and a canonical bijection
\[ \iota_{B,\psi} : \tx{Irr}(C_\phi) \rw \Pi_\phi,\quad \rho \mapsto \pi_\rho, \]
which sends the trivial representation to a $(B,\psi)$-generic representation. This bijection can alse be written as a pairing $\<\> : C_\phi \times \Pi_\phi \rw \C$, and this is the language adopted by Arthur. A semi-simple element $s \in \tx{Cent}(\phi,\hat G)$ gives rise to an endoscopic datum $\mf{e}=(H,\mc{H},s,\xi)$ for $G$. We briefly recall the construction: $\hat H = \hat G_s^\circ$, $\mc{H}=\hat H \cdot \phi(W)$, and $\xi$ is the inclusion map. The group $\mc{H}$ is an extension of $W$ by $\hat H$, and hence provides a homomorphism $\Gamma \rw \tx{Out}(\hat H)$. The group $H$ is the unique quasi-split group with complex dual $\hat H$ for which the homomorphism $\Gamma \rw \tx{Out}(H)$ given by the rational structure coincides under the canonical isomorphism $\tx{Out}(H) \cong \tx{Out}(\hat H)$ with the homomorphism $\Gamma \rw \tx{Out}(\hat H)$ given by $\mc{H}$. In addition to the datum $(H,\mc{H},s,\xi)$, Arthur chooses \cite[\S1.2]{Ar11} an $L$-isomorphism $\xi_{H_1}:\mc{H} \rw {^LH}$. By construction $\phi$ factors through $\xi$ and we obtain $\phi_s = \xi_{H_1}\circ\phi$ which is a Langlands parameter for $H$. Associated to this Langlands parameter is an $L$-packet on $H$, whose stable character we denote by $S\Theta_{\phi_s}$ (this is the stable form (2.2.2) in \cite{Ar11}). Let $\mf{z_e}$ denote the pair $(H,\xi_{H_1})$. This is strictly speaking not a $z$-pair in the sense of \cite[\S2.2]{KS99}, because $H$ will in general not have a simply-connected derived group, but this will not cause any trouble. Let $\Delta[\psi,\mf{e},\mf{z_e}]$ denote the Whittaker normalization of the transfer factor. Arthur shows \cite[Thm. 2.2.1]{Ar11} that if $f \in \mc{H}(G)$ and $f^s \in \mc{H}(H)$ have $\Delta[\psi,\mf{e},\mf{z_e}]$-matching orbital integrals, then
\[ S\Theta_{\phi_s}(f^s) = \sum_{\rho \in \tx{Irr}(C_\phi)}\<s,\rho\>\Theta_{\pi_\rho}(f). \]
The group $C_\phi$ is finite and abelian, and $\tx{Irr}(C_\phi)$ is the set of characters of $C_\phi$, which is also a finite abelian group. Performing Fourier-inversion on these finite abelian groups one obtains
\[ \Theta_{\pi_\rho}(f) = |C_\phi|^{-1}\sum_{s \in C_\phi} \<s,\rho\> S\Theta_{\phi^s}(f^s).\]
This formula is the inversion of endoscopic transfer in the $p$-adic case, for symplectic or odd orthogonal groups.

If $G$ is an even orthogonal group, the following subtle complication occurs: Arthur \cite[Thm. 8.4.1]{Ar11} associates to a given tempered Langlands parameter $\phi$ not one, but two $L$-packets $\Pi_{\phi,1}$ and $\Pi_{\phi,2}$. Each of them comes with a canonical bijection $\iota_{B,\psi,i} : \tx{Irr}(C_\phi) \rw \Pi_{\phi,i}$. There exists a group $\tx{\tilde Out}_N(G)$ of order 2 (independent of $\phi$), and for each $\rho \in \tx{Irr}(C_\phi)$ the two representations $\iota_{B,\psi,1}(\rho)$ and $\iota_{B,\psi,2}(\rho)$ are an orbit under the action of this group. For each $\phi$, there is the following dichotomy: Either $\Pi_{\phi,1}=\Pi_{\phi,2}$, and $\tx{\tilde Out}_N(G)$ acts trivially on this $L$-packet; or $\Pi_{\phi,1} \cap \Pi_{\phi,2} = \emptyset$, and the generator of $\tx{\tilde Out}_N(G)$ sends $\Pi_{\phi,1}$ to $\Pi_{\phi,2}$. In this situation, we will take $\iota_{B,\psi}(\rho)$ to mean the pair of representations $\{\iota_{B,\psi,1}(\rho),\iota_{B,\psi,2}(\rho)\}$. Following Arthur, we will use the notation $\mc{\tilde H}(G)$ to denote denote the subalgebra of $\tx{\tilde Out}_N(G)$-fixed functions in $\mc{H}(G)$ if $G$ is a $p$-adic even orthogonal group. For all other simple groups $G$, we set $\mc{\tilde H}(G)$ equal to $\mc{H}(G)$. If $G$ is a product of simple factors $G_i$, then $\mc{\tilde H}(G)$ is determined by $\mc{\tilde H}(G_i)$. All constructions, as well as the two character identities displayed above, continue to hold, but only for functions $f \in \mc{\tilde H}(G)$. Notice that on $f \in \mc{\tilde H}(G)$, the characters of the two representations $\iota_{B,\psi,1}(\rho)$ and $\iota_{B,\psi,2}(\rho)$ evaluate equally, and moreover $f^s \in \mc{\tilde H}(H)$, so the above character relations do indeed make sense.

We will now describe the analogous formulas in the real case, which are results of Shelstad \cite{Sh08}. Let $G$ be a quasi-split connected reductive group defined over $F=\R$ and fix a Whittaker datum $(B,\psi)$. Let $\phi : W \rw {^LG}$ be a tempered Langlands parameter, and $C_\phi$ as above. One complicating factor in the real case is that, while there is a canonical map
\[ \Pi_\phi \rw \tx{Irr}(C_\phi), \]
it is not bijective, but only injective. It was observed by Adams and Vogan that, in order to obtain a bijective map, one must replace $\Pi_\phi$ by the disjoint union of multiple $L$-packets. All these $L$-packets correspond to $\phi$, but belong to different inner forms of $G$. The correct inner forms to take are the ones parameterized by $H^1(F,G_\tx{sc})$. The disjoint union of these inner forms is sometimes called the $K$-group associated to $G$. For an exposition on $K$-groups we refer the reader to \cite[\S2]{Ar99} and \cite{Sh08}. Writing ${\bf\Pi_\phi}$ for the disjoint union of $L$-packets over all inner forms in the $K$-group, one now has again a bijection \cite[\S11]{Sh08}
\[ {\bf\Pi_\phi} \rw \tx{Irr}(C_\phi), \]
whose inverse we will denote by $\iota_{B,\psi}$, and we denote by $\<\>$ again the pairing between $C_\phi$ and ${\bf \Pi_\phi}$ given by this bijection. Note that in \cite{Sh08}, Shelstad uses a variant of $C_\phi$ involving the simply-connected cover of $\hat G$. Since we are only considering quasi-split $K$-groups (i.e. those which contain a quasi-split form), this variant will not be necessary and the group $C_\phi$ will be enough.

From a semi-simple element $s \in \tx{Cent}(\phi,\hat G)$ we obtain an endoscopic datum $\mf{e}$ by the same procedure as just described. A second complicating factor is that, contrary to $p$-adic case discussed above, there will in general be no $L$-isomorphism $\mc{H} \rw {^LH}$. Instead, one chooses a $z$-extension $H_1$ of $H$. Then there exists an $L$-embedding $\xi_{H_1} : \mc{H} \rw {^LH_1}$. We let $\mf{z_e}$ denote the datum $(H_1,\xi_{H_1})$, and $\phi_s = \xi_{H_1}\circ\phi$. This is now a tempered Langlands parameter for $H_1$. Then Shelstad shows \cite[\S11]{Sh08} that for any two functions ${\bf f}\in \mc{H}({^KG}(F))$ and $f^s \in \mc{H}(H_1(F))$ whose orbital integrals are $\Delta[\lambda,\mf{e},\mf{z_e}]$-matching, one has the formulas
\[ S\Theta_{\phi_s}(f^s) = \sum_{\rho \in \tx{Irr}(C_\phi)}\<s,\rho\>\Theta_{\pi_\rho}({\bf f}), \]
and
\[ \Theta_{\pi_\rho}({\bf f}) = |C_\phi|^{-1}\sum_{s \in C_\phi} \<s,\rho\> S\Theta_{\phi^s}(f^s).\]
In the following sections, we will not use the notation $^KG$ for a $K$-group and the bold-face symbols for objects associated with it. Rather, we will treat it like a regular group and denote it by $G$, in order to simplify the statements of the results.

\section{Change of Whittaker data} \label{sec:chanwhit}

Let $G$ be a quasi-split connected reductive group defined over a real or $p$-adic field $F$. Given a finite abelian group $A$, we will write $A^D$ for its group of characters.

\begin{lem} There exists a canonical injection (bijection, if $F$ is $p$-adic)
\[ G_\tx{ad}(F)/G(F) \rw \tx{ker}(H^1(W,\hat Z_\tx{sc}) \rw H^1(W,\hat Z))^{D}. \]
\end{lem}
\begin{proof}
We will write $G(F)^{\tilde D}$ for the set of continuous characters on $G(F)$ which are trivial on the image of $G_\tx{sc}(F)$. Recall that Langlands has constructed surjective homomorphisms $H^1(W,\hat Z) \rw G(F)^{\tilde D}$ and $H^1(W,\hat Z_\tx{sc}) \rw G_\tx{ad}(F)^{\tilde D}$. If $F$ is $p$-adic, they are also bijective and the statement follows right away, because the finite abelian group $G_\tx{ad}(F)/G(F)$ is Pontryagin dual to
\begin{equation} \tx{ker}(G_\tx{ad}(F)^{\tilde D} \rw G(F)^{\tilde D}) \label{eq:gadg} \end{equation}

If $F$ is real, the kernel of $H^1(W,\hat Z_\tx{sc}) \rw G_\tx{ad}(F)^{\tilde D}$ maps onto the kernel of $H^1(W,\hat Z) \rw G(F)^{\tilde D}$ (this is obvious from the reinterpretation of these homomorphisms given in \cite[\S3.5]{Ka12}). This implies that the kernel of
\[ H^1(W,\hat Z_\tx{sc}) \rw H^1(W,\hat Z) \]
surjects onto \eqref{eq:gadg}.
\end{proof}

Let $\mf{w},\mf{w'}$ be two Whittaker data for $G$. We denote by $(\mf{w},\mf{w'})$ the unique element of $G_\tx{ad}(F)/G(F)$ which conjugates $\mf{w}$ to $\mf{w'}$. We view this element as a character on the finite abelian group
\begin{equation} \tx{ker}(H^1(W,\hat Z_\tx{sc}) \rw H^1(W,\hat Z)) \label{eq:kerzhat} \end{equation}
via the above lemma. Given a Langlands parameter $\phi : W' \rw {^LG}$, we consider the composition
\[ H^0(W,\phi,\hat G) \rw H^0(W,\phi,\hat G_\tx{ad}) \rw H^1(W,\hat Z_\tx{sc}), \]
where $H^0(W,\phi,-)$ denotes the set of invariants of $W$ with respect to the action given by $\phi$. This map is continuous, hence it kills the connected component of the algebraic group $H^0(W,\phi,\hat G)$. Furthermore, it kills $H^0(W,\hat Z)$. Thus we obtain a map
\[ \pi_0(S_\phi/Z(\hat G)^\Gamma) \rw \tx{ker}(H^1(W,\hat Z_\tx{sc}) \rw H^1(W,\hat Z)). \]
In this way, $(\mf{w},\mf{w'})$ gives rise to a character on $\pi_0(S_\phi/Z(\hat G)^\Gamma)$, which we again denote by $(\mf{w},\mf{w'})$.

Now let $s \in S_\phi$. Consider the endoscopic datum $\mf{e}=(H,\mc{H},s,\xi)$ given by $\hat H = \hat G_s^\circ$, $\mc{H}=\hat H \cdot \phi(W)$, and $\xi$ the natural inclusion. Let $\mf{z_e}$ be any $z$-pair for $\mf{e}$. We denote by $\Delta[\mf{w},\mf{e},\mf{z_e}]$ the Langlands-Shelstad transfer factor \cite{LS87}, normalized with respect to $\mf{w}$ (whose definition we will briefly recall in the following proof).

\begin{lem} \label{lem:chanwhit}
\[ \Delta[\mf{w'},\mf{e},\mf{z_e}] = \< (\mf{w},\mf{w'}),s\> \cdot \Delta[\mf{w},\mf{e},\mf{z_e}]. \]
\end{lem}
\begin{proof}
Write $\mf{w}=(B,\psi)$. Let $\tb{spl}=(T,B,\{X_\alpha\})$ be a splitting of $G$ containing the Borel subgroup $B$ given by $\mf{w}$ and $\psi_F : F \rw \C^\times$ be a character with the property that $\tx{spl}$ and $\psi_F$ give rise to $\psi$ as in \cite[\S5.3]{KS99}. Then $\Delta[\mf{w},\mf{e},\mf{z_e}]$ is defined as the product
\[ \epsilon(V_{G,H},\psi_F) \cdot \Delta[\tx{spl},\mf{e},\mf{z_e}], \]
where $\Delta[\tx{spl},\mf{e},\mf{z_e}]$ is the normalization of the transfer factor relative to the splitting $\tx{spl}$ as constructed in \cite[\S3.7]{LS87} (where it is denoted by $\Delta_0$), and $\epsilon(V_{G,H},\psi_F)$ is the epsilon factor (with Langlands' normalization, see e.g. \cite[(3.6)]{Tat79}) of the degree-zero virtual $\Gamma$ representation
\[ V_{G,H} = X^*(T)\otimes\C - X^*(T^H)\otimes \C \]
where $T^H$ is any quasi-split maximal torus of $H$.

Let $g \in G_\tx{ad}(F)$ be an element with $\tx{Ad}(g)\mf{w}=\mf{w'}$. Put $\tx{spl'}=\tx{Ad}(g)\tx{spl}$. Then $\tx{spl'}$ and $\psi_F$ give rise to the Whittaker datum $\mf{w'}$, and consequently we have
\[ \Delta[\mf{w'},\mf{e},\mf{z_e}] = \epsilon(V_{G,H},\psi_F) \cdot \Delta[\tx{spl'},\mf{e},\mf{z_e}]\]
Let $z = g^{-1}\sigma(g) \in H^1(F,Z(G_\tx{sc}))$. Choose any maximal torus $S$ of $G$ coming from $H$. According to \cite[\S2.3]{LS87}, we have
\[ \Delta[\tx{spl'},\mf{e},\mf{z_e}] = \<z,s\> \Delta[\tx{spl},\mf{e},\mf{z_e}], \]
where we have mapped $z$ under $H^1(F,Z(G_\tx{sc})) \rw H^1(F,S_\tx{sc})$ and $s$ under $Z(\hat H)^\Gamma \rw \hat S^\Gamma \rw [\hat S_\tx{ad}]^\Gamma$ and paired them using Tate-Nakayama duality. The number $\<z,s\>$ can also be obtained by mapping $s$ under
\[ H^0(W,Z(\hat H)) \rw H^0(W,\phi,\hat G) \rw H^0(W,\phi,\hat G_\tx{ad}) \rw H^1(W,Z(\hat G_\tx{sc})) \]
and pairing it directly with $z$, using the duality between
\[ H^1(F,Z(G_\tx{sc}))=H^1(F,S_\tx{sc} \rw S_\tx{ad})\quad\tx{and}\quad H^1(W,Z(\hat G_\tx{sc}))=H^1(W,S_\tx{sc}\rw S_\tx{ad}). \]
Using \cite[\S3.5]{Ka12}, one sees that this is the same as the number $\<(\mf{w},\mf{w'}),s\>$.
\end{proof}

\begin{thm} Let $G$ be a quasi-split real $K$-group, or a quasi-split symplectic or special orthogonal $p$-adic group. For any tempered Langlands parameter $\phi : W' \rw {^LG}$ and every $\rho \in \tx{Irr}(C_\phi)$, we have
\[ \iota_\mf{w'}(\phi,\rho) = \iota_\mf{w}(\phi,\rho\otimes(\mf{w},\mf{w'})^{-1}). \]
\end{thm}
\begin{proof}
Fix a semi-simple $s \in S_\phi$. As described in Section \ref{sec:recall}, the pair $(\phi,s)$ gives rise to an endoscopic datum $\mf{e}$, and after a choice of a $z$-pair $\mf{z_e}=(H_1,\xi_{H_1})$ for $\mf{e}$, it further gives rise to a Langlands parameter $\phi_s$ for $H_1$. If the functions $f \in \mc{\tilde H}(G)$ and $f^s \in \mc{\tilde H}(H_1(F))$ have $\Delta[\mf{w},\mf{e},\mf{z_e}]$-matching orbital integrals, then by Lemma \ref{lem:chanwhit} the functions $f$ and $\<(\mf{w},\mf{w'}),s\>\cdot f^s$ have $\Delta[\mf{w'},\mf{e},\mf{z_e}]$-matching orbital integrals. Thus we have
\begin{eqnarray*}
\sum_\rho \<s,\rho\>\Theta_{\iota_\mf{w'}(\phi,\rho)}(f)&=&\<(\mf{w},\mf{w'}),s\>S\Theta_{\phi^s}(f^s)\\
&=&\<(\mf{w},\mf{w'}),s\>\sum_\rho \<s,\rho\>\Theta_{\iota_\mf{w}(\phi,\rho)}(f)\\
&=&\sum_\rho \<s,\rho\otimes (\mf{w},\mf{w'})\>\Theta_{\iota_\mf{w}(\phi,\rho)}(f)\\
&=&\sum_\rho \<s,\rho\>\Theta_{\iota_\mf{w}(\phi,\rho\otimes (\mf{w},\mf{w'})^{-1})}(f),\\
\end{eqnarray*}
where the sums run over $\rho \in \tx{Irr}(C_\phi)$. Since this is true for all $s$, Fourier-inversion gives the result.
\end{proof}

\section{Tempered representations and their contragredient} \label{sec:cont}

In this section, we will prove formula \eqref{eq:expcont} for quasi-split real $K$-groups and quasi-split $p$-adic symplectic and special orthogonal groups. The bulk of the work lies in an analysis of some properties of transfer factors, which is what we turn to now.

Let $F$ be any local field of characteristic zero and $G$ a quasi-split connected reductive group over $F$. We fix an $F$-splitting $\tx{spl}=(T,B,\{X_\alpha\})$ of $G$. We write $\hat G$ for the complex dual of $G$ and fix a splitting $\hat{\tx{spl}}=(\hat T, \hat B,\{X_{\hat\alpha}\})$. We assume that the action of $\Gamma$ on $\hat G$ preserves $\hat{\tx{spl}}$, and that there is an isomorphism $X_*(T) \cong X^*(\hat T)$ which identifies the $B$-positive cone with the $\hat B$-positive cone. Let $\hat C$ be the Chevalley involution on $\hat G$ which sends $\hat{\tx{spl}}$ to the opposite splitting \cite[\S2]{AV12}. The automorphism $\hat C$ commutes with the action of $\Gamma$ and thus $^LC=\hat C \times \tx{id}_W$ is an $L$-automorphism of $^LG$.

Given a torus $S$ defined over $F$, we will denote by $-1$ the homomorphism $S \rw S$ which sends $s \in S$ to $s^{-1}$. It is of course defined over $F$. Its dual ${^LS} \rw {^LS}$ is given by $(s,w) \mapsto (s^{-1},w)$ and will also be denoted by $-1$. Given a maximal torus $S \subset G$ and a set of $\chi$-data $X=\{\chi_\alpha|\alpha \in R(S,G)\}$ for $R(S,G)$ \cite[\S2.5]{LS87}, we denote by $-X$ the set $\{\chi_\alpha^{-1}|\alpha \in R(S,G)\}$. This is also a set of $\chi$-data.

\begin{lem} \label{lem:chiinv}Let $S \subset G$ be a maximal torus defined over $F$, and let $X$ be $\chi$-data for $R(S,G)$. Let $^L\xi_X : {^LS} \rw {^LG}$ be the canonical $\hat G$-conjugacy class of embeddings associated to $X$. Then we have the diagram
\[ \xymatrix{ {^LG}\ar[r]^{^LC}&{^LG}\\ {^LS}\ar[u]^{^L\xi_X}\ar[r]^{-1}&{^LS}\ar[u]_{^L\xi_{-X}} } \]
\end{lem}
\begin{proof} To choose a representative $\xi_X$ within its $\hat G$-conjugacy class, we follow the constructions in \cite[\S 2.6]{LS87}. We choose a Borel subgroup defined over $\ol{F}$ and containing $S$. This provides an admissible isomorphism $\hat\xi : \hat S \rw \hat T$. For $w \in W$, let $\sigma_S(w) \in \Omega(\hat T,\hat G)$ be defined by
\[ \hat\xi( {^ws} ) = {^{\sigma_S(w)w}\hat\xi(s)}. \]
Then a representative of ${^L\xi_X}$ is given by
\[ {^L\xi_X}(s,w) = [\hat\xi(s)r_{\hat B,X}(w) n(\sigma_S(w)),w] \]
Using the fact that $\hat C$ acts by inversion on $\hat T$ and \cite[Lemma  5.8]{AV12}, we see
\[ {^LC}\circ{^L\xi_X}(s,w) = [\hat\xi(s)^{-1}r_{\hat B,X}(w)^{-1} n(\sigma_S(w)^{-1})^{-1},w] \]
One sees that $r_{\hat B,X}(w)^{-1}=r_{\hat B,-X}(w)$. Moreover, by Lemma 5.4 of loc. cit we have
\[ n(\sigma_S(w)^{-1})^{-1} = [t\cdot \sigma_S(w)t^{-1}] n(\sigma_S(w)), \]
where $t \in \hat T$ is any lift of $\rho^\vee(-1) \in \hat T_\tx{ad}$, $\rho^\vee$ being half the sum of the positive coroots. We can choose $t \in \hat T^\Gamma$ by choosing a root $i$ of $-1$ and putting $t = \prod_{\alpha \in R(\hat T,\hat B)}\alpha^\vee(i)$. Then we see
\[ {^LC}\circ{^L\xi_X}(s,w) = \tx{Ad}(t)\circ{^L\xi_{-X}}\circ(-1)(s,w). \]
\end{proof}

Consider a collection $c=(c_\alpha)_{\alpha \in R(T,G)}$ of elements of $\ol{F}^\times$ which is invariant under $\Omega(T,G) \rtimes \Gamma$. Then $(T,B,\{c_\alpha X_\alpha\})$ is another $F$-splitting of $G$, which we will denote by $c\cdot \tx{spl}$. Given any maximal torus $S \subset G$ and any Borel subgroup $B_S$ containing $S$ and defined over $\ol{F}$, the admissible isomorphism $T \rw S$ which sends $B$ to $B_S$ transports $c$ to a collection $(c_\alpha)_{\alpha \in R(S,G)}$ which is invariant under $\Omega(S,G) \rtimes \Gamma$. Moreover, this collection is independent of the choice of $B_S$ (and also of $B$). If $A=(a_\alpha)_{\alpha \in R(S,G)}$ is a set of $a$-data for $R(S,G)$ \cite[\S2.2]{LS87}, then $c\cdot A=(c_\alpha a_\alpha)_{\alpha \in R(S,G)}$ is also a set of $a$-data.

\begin{lem} \label{lem:splcng} With the above notation, we have
\[  \lambda(S,A,c \cdot \tx{spl}) = \lambda(S,c\cdot A,\tx{spl}) ,\]
where $\lambda$ denotes the splitting invariant constructed in \cite[\S 2.3]{LS87}.
\end{lem}
\begin{proof} We begin by recalling the construction of the splitting invariant.
For a simple root $\alpha \in R(T,G)$, let $\eta^{\tx{spl}}_\alpha : \tx{SL}_2 \rw G$ be the homomorphism determined by the splitting $\tx{spl}$. We put
\[ n_{\tx{spl}}(s_\alpha) = \eta^{\tx{spl}}_\alpha\begin{pmatrix} 0&1\\-1&0\end{pmatrix}. \]
For any $w \in \Omega(T,G)$ choose a reduced expression $w=s_{\alpha_1} \cdots s_{\alpha_n}$ and set
\[ n_{\tx{spl}}(w) = n_{\tx{spl}}(s_{\alpha_1}) \cdots n_{\tx{spl}}(s_{\alpha_n}). \]
This product is independent of the choice of reduced expression.

We choose a Borel subgroup $B_S \subset G$ defined over $\ol{F}$ and containing $S$, and an element $h \in G(\ol{F})$ such that $\tx{Ad}(h)(T,B)=(S,B_S)$. Then, for $\sigma \in \Gamma$ and $s \in S$, we have
\[ \tx{Ad}(h^{-1})[{^\sigma s}] ={^{w_S(\sigma)\sigma}\tx{Ad}(h^{-1})[s]} \]
for some $w_S(\sigma) \in \Omega(T,G)$. Then $\lambda(S,A,\tx{spl}) \in H^1(F,S)$ is the element whose image under $\tx{Ad}(h^{-1})$ is represented by the cocycle
\begin{equation} \sigma \mapsto \prod_{\substack{\alpha \in R(T,G)\\ \alpha>0\\ w_S(\sigma)\alpha<0}} \alpha^\vee(a_{\tx{Ad}(h)\alpha}) \cdot n_{\tx{spl}}(w_S(\sigma)) \cdot \sigma(h^{-1})h, \label{eq:splinv} \end{equation}

We now examine the relationship between $n_{\tx{spl}}$ and $n_{c\cdot\tx{spl}}$. Recall the standard triple
\[ E=\begin{bmatrix}0&1\\0&0\end{bmatrix}\quad H=\begin{bmatrix}1&0\\0&-1\end{bmatrix}\quad F=\begin{bmatrix}0&0\\1&0\end{bmatrix} \]
in $\tx{Lie}(\tx{SL}_2)$. The differential $d\eta_{\tx{spl}}$ sends $(E,H,F)$ to $(X_\alpha,H_\alpha,X_{-\alpha})$, where $H_\alpha=d\alpha^\vee(1)$ and $X_{-\alpha} \in \mf{g}_{-\alpha}$ is determined by $[X_\alpha,X_{-\alpha}]=H_\alpha$. On the other hand, the differential $d\eta_{c\tx{spl}}$ sends $(E,H,F)$ to $(c_\alpha\cdot X_\alpha,H_\alpha,c_\alpha^{-1}X_{-\alpha})$. Thus
\[ \eta_{c\cdot\tx{spl}} = \eta_{\tx{spl}}\circ \tx{Ad}\begin{bmatrix} \sqrt{c_\alpha}&0\\ 0&\sqrt{c_\alpha}^{-1}\end{bmatrix} \]
for an arbitrary choice of a square root of $c_\alpha$. It follows that
\[ n_{c\cdot\tx{spl}}(s_\alpha)=\alpha^\vee(c_\alpha)\cdot n_{\tx{spl}}(s_\alpha). \]
Using induction and \cite[Ch.6, \S1, no. 6, Cor 2]{Bou02}, we conclude that
\[ n_{c\cdot\tx{spl}}(w) = \prod_{\substack{\alpha \in R(T,G)\\ \alpha>0\\ w_S(\sigma)\alpha<0}} \alpha^\vee(c_\alpha) \cdot n_{\tx{spl}}(w). \]
The statement follows by comparing the formula \eqref{eq:splinv} for $\lambda(S,A,c\cdot\tx{spl})$ and $\lambda(S,c\cdot A,\tx{spl})$.
\end{proof}

Let $\theta$ be an automorphism which preserves \tx{spl}, and let $\tb{a} \in H^1(W,Z(\hat G))$. The class $\tb{a}$ corresponds to a character $\omega : G(F) \rw \C^\times$. Let $\hat \theta$ be the automorphism dual to $\theta$, which preserves $\hat{\tx{spl}}$. Note that $\hat\theta$ commutes with the action of $\Gamma$. We will write $^L\theta$ for $\theta \times \tx{id}_W$.

Let $\hat G^1$ be the connected component of the group $\hat G^{\hat\theta}$, let $\hat T^1=\hat T \cap \hat G^1$ and $\hat B^1=\hat B \cap \hat G^1$. Then $\hat G^1$ is a reductive group and $(\hat T^1,\hat B^1)$ is a Borel pair for it. The set $\Delta(\hat T^1,\hat B^1)$ of $\hat B^1$-simple roots for $\hat T^1$ is the set of restrictions to $\hat T^1$ of the set $\Delta(\hat T,\hat B)$ of $\hat B$-simple roots for $\hat T$. Moreover, the fibers of the restriction map
\[ \tx{res}: \Delta(\hat T,\hat B) \rw \Delta(\hat T^1,\hat B^1) \]
are precisely the $\Gamma$-orbits in $\Delta(\hat T,\hat B)$. We will write $\alpha_\tx{res}$ for the image of $\alpha$ under $\tx{res}$.

We can extend the pair $(\hat T^1,\hat B^1)$ to a $\Gamma$-splitting $\hat{\tx{spl}}^1=(\hat T^1,\hat B^1,\{X_{\alpha_\tx{res}}\})$ of $\hat G^1$ by setting for each $\alpha_\tx{res} \in \Delta(\hat T^1,\hat B^1)$
\[ X_{\alpha_\tx{res}} = \sum_{\substack{\beta \in \Delta(\hat T,\hat B) \\ \beta_\tx{res}=\alpha_\tx{res}}} X_\beta. \]

Since $\hat\theta$ commutes with $\Gamma$, the group $\hat G^1$ and the splitting just constructed is preserved by $\Gamma$. Thus, $\hat G^1 \rtimes W$ is the $L$-group of a connected reductive group $G^1$. Moreover, since $\hat\theta$ also commutes with $\hat C$, the automorphism $\hat C$ preserves the group $\hat G^1$ and acts by inversion on its maximal torus $\hat T^1$. Thus, $\hat C$ is a Chevalley involution for $\hat G^1$. However, it is not true that $\hat C$ sends the splitting $\hat{\tx{spl}}^1$ to its opposite. Rather, it sends $\hat{\tx{spl}}^1$ to the splitting of $\hat G^1$ constructed from the opposite of $\hat{\tx{spl}}$ by the same procedure as above. That this splitting differs from the opposite of $\hat{\tx{spl}}^1$ is due to the fact that for $\alpha_\tx{res} \in R(\hat T^1,\hat G^1)$, the coroot $H_{\alpha_\tx{res}}$ is not always the sum of $H_\beta$ for all $\beta$ in the $\Gamma$-orbit corresponding to $\alpha_\tx{res}$. In fact, we have
\[ H_{\alpha_\tx{res}} = c_{\alpha_\tx{res}} \cdot \sum_{\substack{\beta \in \Delta(\hat T,\hat B)\\ \beta_\tx{res}=\alpha_\tx{res}}} H_\beta \]
where $c_\alpha=1$ if $\alpha_\tx{res}$ is of type $R_1$, and $c_\alpha=2$ if $\alpha_\tx{res}$ is of type $R_2$. Nevertheless, since both the splitting opposite to $\hat{\tx{spl}}^1$ and the splitting $\hat C(\hat{\tx{spl}}^1)$ are fixed by $\Gamma$, there exists an element of $[\hat G^1]^\Gamma$ that conjugates the one to the other. In other words, ${^LC}$ is $\hat G^1$-conjugate to the Chevalley involution on $^LG^1$ determined by the splitting $\hat{\tx{spl}}^1$.

Given an endoscopic datum $\mf{e}=(H,s,\mc{H},\xi)$ for $(G,\theta,\tb{a})$, we write ${^LC}(\mf{e})$ for the quadruple $(H,\hat C(s^{-1}),\mc{H},{^LC}\circ{^L\theta}\circ\xi)$. Given a z-pair $\mf{z_e}=(H_1,\xi_{H_1})$ for $\mf{e}$, put ${^LC^H}(\mf{z_e})=(H_1,{^LC^H}\circ\xi_{H_1})$, where $^LC^H$ is the Chevalley involution on $^LH_1$.

\begin{fct} The quadruple ${^LC}(\mf{e})$ is an endoscopic datum for $(G,\theta^{-1},\tb{a})$, and ${^LC^H}(\mf{z_e})$ is a $z$-pair for it. If $\mf{e}'$ is an endoscopic datum for $(G,\theta,\tb{a})$ equivalent to $\mf{e}$, then ${^LC}(\mf{e}')$ is equivalent to ${^LC}(\mf{e})$. \end{fct}
\begin{proof} Straightforward. \end{proof}

Let $(B,\psi)$ be a $\theta$-stable Whittaker datum for $G$. Then, associated to $(G,\theta,\tb{a})$, $(B,\psi)$, $\mf{e}$, and $\mf{z_e}$, we have the Whittaker normalization of the transfer factor for $G$ and $H_1$. In fact, as explained in \cite[\S5.5]{KS12}, there are two different such normalizations -- one adapted to the classical local Langlands correspondence for tori \cite[(5.5.2)]{KS12}, and one adapted to the renormalized correspondence \cite[(5.5.1)]{KS12}. To be consistent with the notation chosen in \cite{KS12}, we will call these transfer factors $\Delta'[\psi,\mf{e},\mf{z_e}]$ (for the classical local Langlands correspondence), and $\Delta_D[\psi,\mf{e},\mf{z_e}]$ (for the renormalized correspondence). On the other hand, associated to $(G,\theta^{-1},\tb{a})$, $(B,\psi^{-1})$, $^LC(\mf{e})$, and $^LC^H(\mf{z_e})$, we also have the Whittaker normalization of the transfer factor, again in the two versions. We will call these $\Delta'[\psi^{-1},{^LC}(\mf{e}),{^LC^H}(\mf{z_e})]$ and $\Delta_D[\psi^{-1},{^LC}(\mf{e}),{^LC^H}(\mf{z_e})]$. In the case $\theta=1$ and $\tb{a}=1$ (i.e. ordinary endoscopy), one also has the normalizations $\Delta$ and $\Delta'_D$ \cite[\S5.1]{KS12}. The normalization $\Delta$ is the one compatible with \cite{LS87}.

\begin{pro} \label{pro:deltainv} Let $(B,\psi)$ be a $\theta$-stable Whittaker datum for $G$. Let $\gamma_1 \in H_1(F)$ be a strongly $G$-regular semi-simple element, and let $\delta \in G(F)$ be a strongly-$\theta$-regular $\theta$-semi-simple element. We have
\[ \Delta'[\psi,\mf{e},\mf{z_e}](\gamma_1,\delta) = \Delta'[\psi^{-1},{^LC}(\mf{e}),{^LC^H}(\mf{z_e})](\gamma_1^{-1},\theta^{-1}(\delta^{-1})). \]
The same equality holds with $\Delta_D$ in place of $\Delta'$. Moreover, in the setting of ordinary endoscopy, the equality also holds for $\Delta$ and $\Delta'_D$.
\end{pro}
\begin{proof}
Let us fist discuss the different versions of the transfer factor. In ordinary endoscopy, one obtains $\Delta$ from $\Delta'$ by replacing $s$ with $s^{-1}$. Thus it is clear that the above equality will hold for the one if and only if it holds for the other. The same is true for $\Delta_D$ and $\Delta'_D$. Returning to twisted endoscopy, the difference between $\Delta'$ and $\Delta_D$ is more subtle, and the statement for the one does not formally follow from the statement for the other. However, the proof for both cases is the same, and we will give it for the case of $\Delta_D$.

One sees easily that $\gamma$ is a $\theta$-norm of $\delta$ precisely when $\gamma^{-1}$ is a $\theta^{-1}$-norm of $\theta^{-1}(\delta^{-1})$. We assume that this is the case. Let $S_{H_1} \subset H_1$ be the centralizer of $\gamma_1$, let $S_H \subset H$ be the image of $S_{H_1}$. The torus $S_H$ is the centralizer of the image $\gamma \in H(F)$ of $\gamma_1$. We choose a $\theta$-admissible maximal torus $S \subset G$, an admissible isomorphism $\phi : S_H \rw S_\theta$, an element $\delta^* \in S(\ol{F})$ whose image in $S_\theta$ equals $\phi(\gamma)$, and an element $g \in G_\tx{sc}(\ol{F})$ with $\delta^*=g\delta\theta(g^{-1})$. Then
\[ \theta^{-1}(\delta^{*-1})= g \cdot \theta^{-1}(\delta^{-1}) \cdot \theta^{-1}(g^{-1}). \]

To analyze the transfer factor, we choose $\theta$-invariant $a$-data A for $R(S,G)$ and $\chi$-data X for $R_\tx{res}(S,G)$. Moreover, we fix an additive character $\psi_F : F \rw \C^\times$ and assume that the splitting $\tx{spl}=(T,B,{X_\alpha})$ and the character $\psi_F$ give rise to the fixed Whittaker datum $(B,\psi)$.

Up to equivalence of endoscopic data we may assume $s \in \hat T$. Then $\hat C(s^{-1})=s$, so that the difference between $\mf{e}$ and ${^LC}(\mf{e})$ is only in the embedding $\mc{H} \rw {^LG}$, which changes from $\xi$ to ${^LC}\circ{^L\theta}\circ\xi$.

Then we have \cite[(5.5.1)]{KS12}
\begin{equation} \label{eq:dprod} \Delta_D[\psi,\mf{e},\mf{z_e}]  = \epsilon(V_{G,H},\psi_F) \cdot \Delta_I^\tx{new}[\tx{spl},A] \cdot \Delta_{II}^{-1}[A,X]\cdot \Delta_{III}[\mf{e},\mf{z_e},X]\cdot \Delta_{IV}.\end{equation}
The factor $\epsilon(V_{G,H},\psi_F)$ is the epsilon factor (with Langlands' normalization, see e.g. \cite[(3.6)]{Tat79}) for the virtual $\Gamma$-representation
\[ V_{G,H} = X^*(T)^\theta \otimes \C - X^*(T^H)\otimes \C, \]
where $T^H$ is any quasi-split maximal torus of $H$. It does not depend on any further data.

The factors $\Delta_I$ through $\Delta_{IV}$ depend on most of the objects chosen so far. We have indicated in brackets the more important objects on which they depend, as it will be necessary to keep track of them. These are not all the dependencies. For example, all factors $\Delta_i$ depend on the datum $\mf{e}$, but except for $\Delta_{III}$, this dependence is only through the datum $s$, which we have arranged to be equal for $\mf{e}$ and ${^LC}(\mf{e})$, and so we have not included $\mf{e}$ in the notation for these factors.

We now examine the factors $\Delta_I$ through $\Delta_{IV}$, with the factor $\Delta_{III}$ requiring the bulk of the work. For the factor $\Delta_{IV}$, we have
\begin{equation} \label{eq:d4inv} \Delta_{IV}(\gamma_1^{-1},\theta^{-1}(\delta^{-1})) = \Delta_{IV}(\gamma_1,\delta) \end{equation}
because multiplication by $-1$ on $R(S,G)$ preserves the $\theta$-orbits as well as their type ($R_1$,$R_2$, or $R_3$).

The factor $\Delta_I^\tx{new}[\mf{e},\tx{spl},A]$ does not depend directly on $\gamma_1$ and $\delta$, but rather only on the choices of $S$ and $\phi$. These choices also serve $\gamma_1^{-1}$ and $\theta^{-1}(\delta^{-1})$, and we see
\begin{equation} \label{eq:d1inv} \Delta_I^\tx{new}[\tx{spl},A](\gamma_1^{-1},\theta^{-1}(\delta^{-1})) = \Delta_I^\tx{new}[\tx{spl},A](\gamma_1,\delta). \end{equation}
We turn to $\Delta_{II}[A,X]$. Let $-A$ denote the $a$-data obtained from $A$ by replacing each $a_\alpha$ by $-a_\alpha$. Let $-X$ denote the $\chi$-data obtained from $X$ by replacing each $\chi_\alpha$ by $\chi_\alpha^{-1}$. Then one checks that
\begin{equation} \label{eq:d2inv} \Delta_{II}[A,X](\gamma_1^{-1},\theta^{-1}(\delta^{-1})) = \Delta_{II}[-A,-X](\gamma_1,\delta). \end{equation}

Before we can examine $\Delta_{III}[\mf{e},\mf{z_e},X]$, we need to recall its construction, following \cite[\S 4.4, \S 5.3]{KS99}. We define an $F$-torus $S_1$ as the fiber product
\[ \xymatrix{ S_1\ar[d]\ar[rr]&&S\ar[d]\\ S_{H_1}\ar[r]&S_H\ar[r]^{\phi}&S_\theta } \]
The element $\delta_1=(\gamma_1,\delta)$ belongs to $S_1$. The automorphism $\tx{id} \times \theta$ of $S_{H_1} \times S$ induces an automorphism $\theta_1$ of $S_1$. This automorphism restricts trivially to the kernel of $S_1 \rw S$, and hence $1-\theta_1$ induces a homomorphism $S \rw S_1$, which we can compose with $S_\tx{sc} \rw S$ to obtain a homomorphism $S_\tx{sc} \rw S_1$, which we still denote by $1-\theta_1$.

The element $(\sigma(g)g^{-1},\delta_1)$ belongs to $H^1(F,S_\tx{sc} \stackrel{1-\theta_1}{\lrw} S_1)$ and is called $\tx{inv}(\gamma_1,\delta)$. In \cite[A.3]{KS99}, Kottwitz and Shelstad construct a pairing $\<\>_\tx{KS}$ between the abelian groups $H^1(F,S_\tx{sc} \stackrel{1-\theta_1}{\lrw} S_1)$ and $H^1(W,\hat S_1 \stackrel{1-\hat\theta_1}{\lrw} \hat S_\tx{ad})$. Using this pairing, they define
\[ \Delta_{III}[\mf{e},\mf{z_e},X](\gamma_1,\delta) = \<\tx{inv}(\gamma_1,\delta), A_0[\mf{e},\mf{z_e},X]\>_\tx{KS} \]
where $A_0[\mf{e},\mf{z_e},X]$ is an element of $H^1(F,\hat S_1 \stackrel{1-\hat\theta_1}{\lrw} \hat S_\tx{ad})$ constructed as follows:

The $\chi$-data $X$ provides an $\hat H$-conjugacy class of embeddings ${^LS_H} \rw {^LH}$ and a $\hat G^1$-conjugacy class of embeddings ${^LS_\theta} \rw {^LG^1}$, where $\hat G^1$ is the connected stabilizer of $\hat\theta$. Composing with the canonical embeddings ${^LH} \rw {^LH_1}$ and ${^LG^1} \rw {^LG}$, we obtain embeddings $\xi_1 : {^LS_\theta} \rw {^LG}$ and $\xi_{S_H} : {^LS_H} \rw {^LH_1}$. There is a unique embedding $\xi_1^1 : {^LS} \rw {^LG}$ extending $\xi_1$, and there is a unique embedding $\xi_{S_H}^1 : {^LS_{H_1}} \rw {^LH_1}$ extending $\xi_{S_H}$.

Define $\mc{U} = \{ x \in \mc{H}| \tx{Ad}(\xi(x))|_{\hat T^1} = \tx{Ad}(\xi_1(1 \times \bar x))|_{\hat T^1} \}$. One can show that $\xi(\mc{U}) \subset \xi_1^1({^LS})$ and $\xi_{H_1}(\mc{U}) \subset \xi_{S_H}^1({^LS_{H_1}})$. Then we can define, for any $w \in W$, an element $a_S[X](w) \in \hat S_1$, by choosing a lift $u(w) \in \mc{U}$ and letting $a_S[X](w)=(s_1^{-1},s) \in \hat S_{H_1} \times \hat S \srw \hat S_1$, where $s_1 \in \hat S_{H_1}$ and $s \in \hat S$ are the unique elements satisfying
\begin{equation} \label{eq:xi0} \xi_1^1(s \times w) = \xi(u(w)) \qquad\tx{and}\qquad \xi_{S_H}^1(s_1 \times w) = \xi_{H_1}(u(w)). \end{equation}
We can further define $s_S = [\xi_1^1]^{-1}(s) \in \hat S$ and also view it as an element of $\hat S_\tx{ad}$. Then
\[ A_0[\mf{e},\mf{z_e},X] = (a_S[X]^{-1},s_S) \in H^1(W,\hat S_1 \rw \hat S_\tx{ad}).\]

We are now ready to examine $\Delta_{III}[\mf{e},\mf{z_e},X]$. We have
\begin{equation} \label{eq:invinv} \tx{inv}(\gamma_1^{-1},\theta^{-1}(\delta^{-1})) = (\sigma(g)g^{-1},\theta_1^{-1}(\delta_1^{-1})). \end{equation}
This is an element of $H^1(F,S_\tx{sc} \stackrel{1-\theta_1^{-1}}{\lrw} S_1)$. We have
\begin{eqnarray*}
&&\Delta_{III}[{^LC}(\mf{e}),{^LC^H}(\mf{e_z}),X](\gamma_1^{-1},\theta^{-1}(\delta^{-1}))\\
&=&\<\tx{inv}(\gamma_1^{-1},\theta^{-1}(\delta^{-1})), A_0[{^LC}(\mf{e}),{^LC^H}(\mf{e_z}),X] \>_\tx{KS}.
\end{eqnarray*}
Here $A_0[{^LC}(\mf{e}),{^LC^H}(\mf{e_z}),X]$ is the element of $H^1(W,\hat S_1 \stackrel{1-\hat\theta_1^{-1}}{\lrw} \hat S_\tx{ad})$, constructed as above, but with respect to the endoscopic datum ${^LC}(\mf{e})$ and the z-pair ${^LC^H}(\mf{z_e})$, rather than $\mf{e}$ and $\mf{z_e}$. Thus $A_0[{^LC}(\mf{e}),{^LC^H}(\mf{e_z}),X] = (\tilde a_S[X]^{-1},s_S)$, with $\tilde a_S[X](w)=(\tilde s_1^{-1},\tilde s)$, and
\[ \xi_1^1[X](\tilde s \times w) = {^L\theta}\circ{^LC}\circ\xi(u(w)) \quad\tx{and}\quad \xi_{S_H}^1[X](\tilde s_1 \times w) = {^LC^H}\circ\xi_{H_1}(u(w)).\]
Using Equation \eqref{eq:xi0} we see
\[ \xi_1^1[X](\tilde s \times w) = {^L\theta}\circ{^LC}\circ\xi_1^1[X](s \times w) \quad\tx{and}\quad \xi_{S_H}^1[X](\tilde s_1 \times w) = {^LC^H}\circ\xi_{S_H}^1[X](s_1 \times w).\]
According to Lemma \ref{lem:chiinv} this is equivalent to
\[ \xi_1^1[X](\tilde s \times w) = {^L\theta}\circ\xi_1^1[-X](s^{-1} \times w) \quad\tx{and}\quad \xi_{S_H}^1[X](\tilde s_1 \times w) = \xi_{S_H}^1[-X](s_1^{-1} \times w).\]
We conclude that
\begin{equation} \label{eq:a0inv} \tilde a_S[X](w) = \hat\theta_1(a_S[-X](w)^{-1}). \end{equation}
The isomorphism of complexes
\[ \xymatrix{ S_\tx{sc}\ar[d]_{1-\theta_1^{-1}}\ar[r]&S_\tx{sc}\ar[d]^{1-\theta_1}\\ S_1\ar[r]^{\theta_1\circ(\ )^{-1} }&S_1} \]
induces an isomorphism $H^1(F,S_\tx{sc} \stackrel{1-\theta_1^{-1}}{\lrw} S_1) \rw H^1(F,S_\tx{sc} \stackrel{1-\theta_1}{\lrw} S_1)$ which, according to Equation \eqref{eq:invinv}, sends $\tx{inv}(\gamma_1^{-1},\theta^{-1}(\delta^{-1}))$ to $\tx{inv}(\gamma_1,\delta)$. The dual isomorphism of complexes
\[ \xymatrix{ \hat S_1 \ar[d]_{1-\hat\theta_1^{-1}}&\hat S_1\ar[d]^{1-\hat\theta_1}\ar[l]_{\theta_1\circ(\ )^{-1} }\\ \hat S_\tx{ad}&\hat S_\tx{ad}\ar[l]  } \]
induces an isomorphism  $H^1(W,\hat S_1 \stackrel{1-\hat\theta_1^{-1}}{\lrw} \hat S_\tx{ad}) \rw  H^1(W,\hat S_1 \stackrel{1-\hat\theta_1}{\lrw} \hat S_\tx{ad})$ which, according to Equation \eqref{eq:a0inv}, sends $A_0[\mf{e},\mf{z_e},-X]$ to $A_0[{^LC}(\mf{e}),{^LC^H}(\mf{z_e}),X]$. We conclude
\begin{equation} \label{eq:d3inv} \Delta_{III}[{^LC}(\mf{e}),{^LC^H}(\mf{z_e}),X](\gamma_1^{-1},\theta^{-1}(\delta^{-1})) = \Delta_{III}[\mf{e},\mf{z_e},-X](\gamma_1,\delta). \end{equation}

Combining equations \eqref{eq:dprod}, \eqref{eq:d4inv}, \eqref{eq:d1inv}, \eqref{eq:d2inv}, \eqref{eq:d3inv}, we obtain
\begin{eqnarray*}
&&\Delta_D[\psi,{^LC}(\mf{e}),{^LC^H}(\mf{z_e})](\gamma_1^{-1},\theta^{-1}(\delta^{-1}))\\
\\
&=&\epsilon(V_{G,H},\psi_F)\\
&\cdot&\Delta_I^\tx{new}[\tx{spl},A](\gamma_1^{-1},\theta^{-1}(\delta^{-1}))\\
&\cdot&\Delta_{II}^{-1}[A,X](\gamma_1^{-1},\theta^{-1}(\delta^{-1}))\\
&\cdot&\Delta_{III}[{^LC}(\mf{e}),{^LC^H}(\mf{z_e}),X](\gamma_1^{-1},\theta^{-1}(\delta^{-1}))\\
&\cdot&\Delta_{IV}(\gamma_1^{-1},\theta^{-1}(\delta^{-1}))\\
\\
&=&\epsilon(V_{G,H},\psi_F)\\
&\cdot&\Delta_I^\tx{new}[\tx{spl},A](\gamma_1,\delta)\\
&\cdot&\Delta_{II}^{-1}[-A,-X](\gamma_1,\delta)\\
&\cdot&\Delta_{III}[\mf{e},\mf{z_e},-X](\gamma_1,\delta)\\
&\cdot&\Delta_{IV}(\gamma_1,\delta)\\
\end{eqnarray*}
Since $-X$ and $-A$ are valid choices of $\chi$-data and $a$-data, according to Equation \eqref{eq:dprod} the second product is almost equal to $\Delta_\psi[\psi,\mf{e},\mf{z_e}]$. The only difference is that the $a$-data occurring in $\Delta_I$ is $A$, while the one occurring in  $\Delta_{II}$ is $-A$. Let $-\tx{spl}$ be the splitting $(T,B,\{-X_\alpha\})$. The splitting $-\tx{spl}$ and the character $\psi_F^{-1}$ give rise to the fixed Whittaker datum $(B,\psi)$, just like the splitting $\tx{spl}$ and the character $\psi_F$ did. Then we have
\[ \epsilon(V_{G,H},\psi_F)\cdot\Delta_I[\tx{spl},A] = \epsilon(V_{G,H},\psi_F^{-1})\cdot\Delta_I[-\tx{spl},A] = \epsilon(V_{G,H},\psi_F^{-1})\cdot\Delta_I[\tx{spl},-A], \]
with the first equality following from the argument of \cite[\S 5.3]{KS99}, and the second from Lemma \ref{lem:splcng}. Noting that $\tx{spl}$ and $\psi_F^{-1}$ give rise to the Whittaker datum $(B,\psi^{-1})$, we obtain

\[ \Delta_D[\psi,{^LC}(\mf{e}),{^LC^H}(\mf{z_e})](\gamma_1^{-1},\theta^{-1}(\delta^{-1})) = \Delta_D[\psi^{-1},\mf{e},\mf{z_e}](\gamma_1,\delta). \]

\end{proof}

\begin{cor} \label{cor:deltainv} Let $f \in \mc{H}(G)$ and $f^{H_1} \in \mc{H}(H_1)$ be functions such that the $(\theta^{-1},\omega)$-twisted orbital integrals of $f$ match the stable orbital integrals of $f^{H_1}$ with respect to $\Delta[\psi^{-1},{^LC}(e),{^LC^H}(z_e)]$. Then the $(\theta,\omega)$-twisted orbital integrals of $f\circ\theta^{-1}\circ i$ match the stable orbital integrals of $f^{H_1} \circ i$ with respect to $\tilde\Delta[\psi,e,z_e]$. Here $\tilde\Delta$ stands for any of the two (resp. four) Whittaker normalizations of the transfer factor for twisted (resp. standard) endoscopy, and $i$ is the map on $G(F)$ or $H_1(F)$ sending every element to its inverse.
\end{cor}
\begin{proof}
\begin{eqnarray*}
&&SO(\gamma_1,f^{H_1}\circ i)\\
&=&SO(\gamma_1^{-1},f^{H_1})\\
&=&\sum_{\delta \in G(F)/\theta^{-1}-\sim}\hspace{-20pt}\tilde\Delta[\psi^{-1},{^LC}(e),{^LC^H}(z_e)](\gamma_1^{-1},\delta)O^{\theta^{-1},\omega}(\delta,f)\\
&=&\sum_{\delta \in G(F)/\theta^{-1}-\sim}\hspace{-20pt}\tilde\Delta[\psi^{-1},{^LC}(e),{^LC^H}(z_e)](\gamma_1^{-1},\delta)O^{\theta,\omega}(\theta(\delta^{-1}),f\circ\theta^{-1}\circ i)\\
&=&\sum_{\delta' \in G(F)/\theta-\sim}\hspace{-20pt}\tilde\Delta[\psi,e,z_e](\gamma_1^{-1},\delta')O^{\theta,\omega}(\delta',f\circ\theta^{-1}\circ i)\\
\end{eqnarray*}
The last line follows from Proposition \ref{pro:deltainv}, with the substitution $\delta'=\theta(\delta^{-1})$.
\end{proof}

\begin{fct} \label{fct:dualgen} Let $\pi$ be an irreducible admissible tempered $\theta$-stable representation of $G(F)$, and let $A: \pi \cong \pi\circ\theta$ be the unique isomorphism which preserves a $(B,\psi)$-Whittaker functional. Then the dual map $A^\vee : (\pi\circ\theta)^\vee \rw \pi^\vee$ preserves a $(B,\psi^{-1})$-Whittaker functional.
\end{fct}
\begin{proof} Let $V$ be the vector space on which $\pi$ acts. Since $\pi$ is tempered, it is unitary. Let $\<\cdot,\cdot\>$ be a $\pi$-invariant non-degenerate Hermitian form on $V$. Then
\[ \ol{V} \rw V^\vee,\quad w \mapsto \<\cdot,w\> \]
is a $\pi$-$\pi^\vee$-equivariant isomorphism, and it identifies $A$ with its $\<\cdot,\cdot\>$-adjoint $A^*$. But $A^*=A^{-1}$. Indeed, $(v,w) \mapsto \<Av,Aw\>$ is another $\pi$-invariant scalar product, hence there exists a scalar $c \in \C^\times$ with $\<Av,Aw\>=c\<v,w\>$. On the one hand, since both sides are Hermitian, this scalar must belong to $\R_{>0}$. On the other hand, since $A$ has finite order, $c$ must be a root of unity. Thus $c=1$, which shows $A^*=A^{-1}$. Let $\sigma$ denote complex conjugation. If $\lambda : V \rw \C$ is a $(B,\psi)$-Whittaker functional preserved by $A$, then $\sigma\circ\lambda : \ol{V} \rw \C$ is a $(B,\psi^{-1})$-Whittaker functional preserved by $A^\vee=A^*=A^{-1}$.
\end{proof}

\begin{cor} \label{cor:dualgen} If $\tilde\pi$ is the unique extension of $\pi$ to a representation of $G(F) \rtimes \<\theta\>$ so that $\tilde\pi(\theta)$ is the isomorphism $\pi \rw \pi\circ\theta$ which fixes a $(B,\psi)$-Whittaker functional, then $\tilde\pi^\vee$ is the unique extension of $\pi^\vee$ to a representation of $G(F) \rtimes \<\theta\>$ so that $\tilde\pi^\vee(\theta)$ is the isomorphism $\pi^\vee \rw \pi^\vee\circ\theta$ which fixes a $(B,\psi^{-1})$-Whittaker functional.
\end{cor}

The next statement is a weaker version of \cite[Thm. 7.1(a)]{AV12}. That theorem states that for any Langlands parameter $\phi : W' \rw {^LG}$ for a real connected reductive group $G$ with corresponding $L$-packet $\Pi_\phi$, the set $\{\pi^\vee|\pi \in \Pi_\phi\}$ is also an $L$-packet, and its parameter is ${^LC}\circ\phi$. Assume that $\phi$ is tempered, and denote by $S\Theta_\phi$ the stable character of the $L$-packet $\Pi_\phi$. Then an immediate corollary of the result of Adams and Vogan is that $S\Theta_\phi\circ i = S\Theta_{{^LC}\circ\phi}$. We will now prove this equality for quasi-split symplectic and special orthogonal $p$-adic groups. After that, we will use it to derive Formula \eqref{eq:expcont}. With this formula at hand, we will then derive the precise $p$-adic analog of \cite[Thm. 7.1(a)]{AV12} as a corollary.

\begin{thm} \label{thm:stabcont} Let $H$ be a quasi-split symplectic or special orthogonal group and $\phi : W' \rw {^LH}$ a tempered Langlands parameter. Write $S\Theta_\phi$ for the stable character of the $L$-packet attached to $\phi$. Then we have an equality of linear forms on $\mc{\tilde H}(H)$
\[ S\Theta_{\phi}\circ i = S\Theta_{{^LC^H}\circ\phi}. \]
\end{thm}
\begin{proof}
We recall very briefly the characterizing property of $S\Theta_\phi$, following Arthur \cite[\S1,\S2]{Ar11}.
The group $H$ can be extended to an endoscopic datum $\mf{e}=(H,\mc{H},s,\xi)$ for the triple $(G,\theta,1)$, where $G=\tx{GL}_n$ for a suitable $n$ depending on $H$ and $\theta$ is an outer automorphism of $G$ preserving the standard splitting. Then $\xi\circ\phi$ is a Langlands parameter for $G$. Let $\pi$ be the representation of $G(F)$ assigned to $\phi$ by the local Langlands correspondence \cite{HT01}, \cite{He00}. We have $\pi \cong \pi\circ\theta$. Choose an additive character $\psi_F : F \rw \C^\times$, and let \tx{spl} be the standard splitting of $G$. Then we obtain a Whittaker datum $(B,\psi)$. There is a unique isomorphism $A : \pi \rw \pi\circ\theta$ which preserves one (hence all) $(B,\psi)$-Whittaker functionals. Then we have the distribution
\[ f \mapsto T\Theta_{\xi\circ\phi}^\psi(f) = \tx{tr}\left( v \mapsto \int_{G(F)} f(g)\pi(g)A v dg \right). \]
By construction, $S\Theta_\phi$ is the unique stable distribution on $\mc{\tilde H}(H)$ with the property that
\[ S\Theta_\phi(f^H) = T\Theta_{\xi\circ\phi}^\psi(f) \]
for all $f \in \mc{\tilde H}(G)$ and $f^H \in \mc{\tilde H}(H)$ such that the $(\theta,\omega)$-twisted orbital integrals of $f$ match the stable orbital integrals of $f^H$ with respect to $\Delta'[\psi,\mf{e},\mf{z_e}]$. Here $\mf{z_e}$ stands for the pair $(H,\xi_{H_1})$, where $\xi_{H_1}$ is a suitably chosen isomorphism $\mc{H} \rw {^LH}$ \cite[\S1]{Ar11}.

Now consider the transfer factor $\Delta'[\psi^{-1},{^LC}(\mf{e}),{^LC^H}(\mf{z_e})]$. We have chosen both $\hat C$ and $\hat C^H$ to preserve the standard diagonal torus and act as inversion on it. Moreover the endoscopic element $s$ belongs to that torus. The using ${^LC}(\mf{e}),{^LC^H}(\mf{z_e})$ has the same effect as using $(H,\mc{H},s,{^LC}\circ{^L\theta}\circ\xi\circ{^LC^H}^{-1})$ and the z-pair $\mf{z_e}$. We have ${^L\theta}\circ\xi = \tx{Int}(s^{-1})\xi$, so replacing ${^L\theta}\circ\xi$ by $\xi$ changes the above datum to an equivalent one. In the same way, we have ${^LC}\circ\xi\circ{^LC^H}^{-1}=\tx{Int}(t)\circ\xi$ for some $t \in \hat T$. All in all, up to equivalence, we see that replacing $\mf{e}$ and $\mf{z_e}$ by ${^LC}(\mf{e})$ and ${^LC^H}(\mf{z_e})$ has no effect, and we have obtained the transfer factor $\Delta'[\psi^{-1},\mf{e},\mf{z_e}]$, which we from now on abbreviate to $\Delta'[\psi^{-1}]$.

We have $S\Theta_{^LC^H\circ\phi}(f^H)=T\Theta_{\xi\circ{^LC^H}\circ\phi}(f)$. As we just argued, $\xi\circ{^LC^H}$ is $\hat G$-conjugate to ${^LC}\circ\xi$. Thus, the Galois-representation ${\xi\circ{^LC^H}\circ\phi}$ is the contragredient to the Galois-representation $\xi\circ\phi$. As the local Langlands correspondence for $\tx{GL}_n$ respects the operation of taking the contragredient, Corollary \ref{cor:dualgen} implies
\[ T\Theta_{\xi\circ{^LC^H}\circ\phi}^\psi(f) = T\Theta_{\xi\circ\phi}^{\psi^{-1}}(f\circ\theta^{-1}\circ i). \]
By construction of $S\Theta_{\xi\circ\phi}$, we have
\[ T\Theta_{\xi\circ\phi}^{\psi^{-1}}(f\circ\theta^{-1}\circ i) = S\Theta_{\xi\circ\phi}('f^H) \]
whenever $'f^H$ is an element of $\mc{\tilde H}(H)$ whose stable orbital integrals match the $(\theta,1)$-orbital integrals of $f\circ\theta^{-1}\circ i$ with respect to $\Delta'[\psi^{-1}]$. By Corollary \ref{cor:deltainv}, $f^H\circ i$ is such a function, and we see that the distribution $f \mapsto S\Theta_{\xi\circ\phi}(f^H\circ i)$ satisfies the property that characterizes $S\Theta_{^LC_H\circ\phi}$, hence must be equal to the latter.

\end{proof}

\begin{thm} \label{thm:cont}
Let $G$ be a quasi-split connected reductive real $K$-group or a quasi-split symplectic or special orthogonal $p$-adic group, and let $(B,\psi)$ be a Whittaker datum. Let $\phi : W' \rw {^LG}$ be a tempered Langlands parameter, and $\rho$ an irreducible representation of $C_\phi$. Then
\[ \iota_{B,\psi}(\phi,\rho)^\vee = \iota_{B,\psi^{-1}}({^LC}\circ\phi,[\rho\circ\hat C^{-1}]^\vee). \]
\end{thm}
\begin{proof}
Put $\pi=\iota_{B,\psi}(\phi,\rho)$. For each semi-simple $s \in S_\phi$, let $\mf{e}_s = (H,\mc{H},s,\xi)$ be the corresponding endoscopic datum (see Section \ref{sec:recall}), and choose a z-pair $\mf{z}_s=(H_1,\xi_{H_1})$. We have the Whittaker normalization $\Delta[\psi,\mf{e}_s,\mf{z}_s]$
of the transfer factor compatible with \cite{LS87} (see the discussion before Proposition \ref{pro:deltainv}).

By construction, $\phi$ factors through the $\xi$. Put $\phi_s=\phi\circ\xi_{H_1}$. For any function $f \in \mc{\tilde H}(G)$ let $f^{s,\psi} \in \mc{\tilde H}(H_1)$ denote the transfer of $f$ with respect to the transfer factor $\Delta[\psi,\mf{e}_s,\mf{z}_s]$. While $\phi_s$ and $f^{s,\psi}$ depend on the choice of $\mf{z}_s$, the distribution
\[ f \mapsto S\Theta_{\phi_s}(f^{s,\psi}) \]
does not. As discussed in Section \ref{sec:recall}, we have the inversion of endoscopic transfer
\[ \Theta_\pi(f) = \sum_{s \in C_\phi} \tr\rho(s) S\Theta_{\phi_s}(f^{s,\psi}). \]

Thus, we need to show that
\[ \Theta_{\pi^\vee}(f) = \sum_{s \in C_{{^LC}\circ\phi}} \tr(\rho^\vee(\hat C^{-1}(s))) S\Theta_{[{^LC}\circ\phi]_s}(f^{s,\psi^{-1}}). \]
We can of course reindex the sum as
\[ \sum_{s \in C_\phi} \tr\rho(s) S\Theta_{[{^LC}\circ\phi]_{s'}}(f^{{s'},\psi^{-1}}), \]
with $s'=\hat C(s^{-1})$. The theorem will be proved once we show
\[ S\Theta_{\phi_s}([f\circ i]^{s,\psi}) = S\Theta_{[{^LC}\circ\phi]_{s'}}(f^{{s'},\psi^{-1}}). \]
The endoscopic datum corresponding to ${^LC\circ\phi}$ and $s'$ is precisely ${^LC}(\mf{e})$. We are free to choose any z-pair for it, and we choose ${^LC^H}(\mf{z_e})$. Then $[{^LC\circ\phi}]_{s'}={^LC^H}\circ\phi_s$. The function $f^{s',\psi^{-1}} \in \mc{\tilde H}(H_1)$ is the transfer of $f$ with respect to the data ${^LC}(e)$, ${^LC^H}(\mf{z_e})$, and the Whittaker datum $(B,\psi^{-1})$. By Corollary \ref{cor:deltainv}, the function $f^{s',\psi^{-1}} \circ i$ is the transfer of $f \circ i$ with respect to $\mf{e}$,$\mf{z_e}$, and the Whittaker datum $(B,\psi)$. In other words,
\[ [f\circ i]^{s,\psi} = f^{s',\psi^{-1}} \circ i. \]
The theorem now follows from Theorem \ref{thm:stabcont}.

\end{proof}

We alert the reader that, as was explained in Section \ref{sec:recall}, the symbol $\iota_{B,\psi}(\phi,\rho)$ refers to an individual representation of $G(F)$ in all cases of Theorem \ref{thm:cont}, except possibly when $G$ is an even orthogonal $p$-adic group, in which case Arthur's classification may assign to the pair $(\phi,\rho)$ a pair of representations, rather than an individual representation. In that case, the theorem asserts that if $(\pi_1,\pi_2)$ is the pair of representations associated with $(\phi,\rho)$, then $(\pi_1^\vee,\pi_2^\vee)$ is the pair of representations associated with $\iota_{B,\psi^{-1}}({^LC}\circ\phi,[\rho\circ\hat C^{-1}]^\vee)$.

The following result is the $p$-adic analog of \cite[Thm. 7.1(a)]{AV12}.

\begin{cor} Let $G$ be a quasi-split symplectic or special orthogonal $p$-adic group, and let $\phi : W' \rw {^LG}$ be a tempered Langlands parameter. If $\Pi$ is an $L$-packet assigned to $\phi$, then
\[ \Pi^\vee=\{ \pi^\vee| \pi \in \Pi \} \]
is an $L$-packet assigned to ${^LC}\circ\phi$.
\end{cor}
This corollary to Theorem \ref{thm:cont} is the $p$-adic analog of \cite[Thm. 1.3]{AV12}.
\begin{proof}
When $G$ is either a symplectic or an odd orthogonal group, the statement follows immediately from Theorem \ref{thm:cont}. However, if $G$ is an even orthogonal group, $\Pi$ is one of two $L$-packets $\Pi_1$, $\Pi_2$ assigned to $\phi$, and a priori we only know that the set $\Pi^\vee$ belongs to the union of the two $L$-packets $\Pi_1',\Pi_2'$ assigned to ${^LC}\circ\phi$. We claim that in fact it equals one of these two $L$-packets. Indeed, let $S\Theta$ be the stable character of $\Pi$. This is now a stable linear form on $\mc{H}(G)$, not just on $\mc{\tilde H}(G)$. The linear form $S\Theta\circ i$ is still stable. If $S\Theta_1'$ and $S\Theta_2'$ are the stable characters of $\Pi_1'$ and $\Pi_2'$ respectively, then the restrictions of $S\Theta_1'$ and $S\Theta_2'$ to $\mc{\tilde H}(G)$ are equal, and moreover according to Theorem \ref{thm:stabcont} these restrictions are also equal to the restriction of $S\Theta \circ i$ to $\mc{\tilde H}(G)$. From \cite[Cor. 8.4.5]{Ar11} we conclude that
\[ S\Theta \circ i = \lambda S\Theta_1' + \mu S\Theta_2' \]
for some $\lambda,\mu \in \C$ with $\lambda+\mu=1$. However, each of the three distributions $S\Theta\circ i, S\Theta_1',S\Theta_2'$ is itself a sum of characters of tempered representations. The linear independence of these characters then forces one of the numbers $\lambda,\mu$ to be equal to $1$, and the other to $0$.
\end{proof}

\section{$p$-adic $L$-packets of zero or minimal positive depth} \label{sec:explicit}
In this section we are going to show that the constructions of depth-zero supercuspidal $L$-packets of \cite{DR09} and of epipelagic $L$-packets of \cite{Ka12} satisfy Equation \eqref{eq:expcont}.

Fix a Langlands parameter $\phi : W \rw {^LG}$ which is either TRSELP or epipelagic. We may fix a $\Gamma$-invariant splitting $(\hat T,\hat B,\{X_{\hat\alpha}\})$ of $\hat G$ and arrange that $\hat T$ is the unique torus normalized by $\phi$. We also choose a representative for $\hat C$ which sends the chosen splitting to its opposite and form ${^LC}=\hat C \times \tx{id}_W$. In both constructions the first step is to form the $\Gamma$-module $\hat S$ with underlying abelian group $\hat T$ and $\Gamma$-action given by the composition $\Gamma \rw \Omega(\hat T,\hat G)\rtimes\Gamma$ of $\phi$ and the natural projection $N(\hat T,\hat G) \rw \Omega(\hat T,\hat G)$. Since ${^LC}$ commutes with every element of $\Omega(\hat T,\hat G)\rtimes\Gamma$, the $\Gamma$-module $\hat S$ for ${^LC}\circ\phi$ is the same as the one for $\phi$. The next step is to obtain from $\phi$ a character $\chi : S(F) \rw \C^\times$ by factoring $\phi$ through an $L$-embedding ${^Lj_X}: {^LS}\rw{^LG}$. For the depth-zero case, this is the reinterpretation given in \cite{Ka11}, and the $L$-embedding ${^LS}\rw{^LG}$ is obtained by choosing arbitrary unramified $\chi$-data $X$ for $R(\hat S,\hat G)$. Thus, this $L$-embedding is the same for $\phi$ and ${^LC}\circ\phi$. Writing $\phi_S : W\rw {^LS}$ for the factored parameter, so that $\phi={^Lj_X}\circ\phi_S$, we see from Lemma \ref{lem:chiinv} that
\[ {^LC}\circ\phi = {^Lj_{-X}}\circ (-1) \circ\phi_S. \]
Since $-X$ is another set of unramified $\chi$-data, and it is shown in \cite[\S3.4]{Ka11} that $\phi_S$ is independent of the choice of $X$, we see that $[{^LC}\circ\phi]_S=(-1)\circ\phi_S$. In other words, the character of $S(F)$ constructed from ${^LC}\circ\phi$ is $\chi_S^{-1}$.

We claim that the same is true for epipelagic Langlands parameters. That case is a bit more subtle because ${^Lj_X}$ depends on $\phi$ more strongly -- the $\chi$-data $X$ is chosen based on the restriction of $\phi$ to wild inertia. What we need to show is that if $X$ is chosen for $\phi$, then the choice for ${^LC}\circ\phi$ is $-X$. This however follows right away from the fact that the restriction of ${^LC}\circ\phi$ to wild inertia equals the composition of $(-1)$ with the restriction of $\phi$ to wild inertia.

The third step in the construction of these two kinds of $L$-packets relies on a procedure (different in the two cases) which associates to an admissible embedding $j$ of $S$ into an inner form $G'$ of $G$ a representation $\pi(\chi_S,j)$ of $G'(F)$. We won't recall this procedure -- for our current purposes it will be enough to treat it as a black box. Now let $(B,\psi)$ be a Whittaker datum for $G$. In the depth-zero case, we choose a hyperspecial vertex in the apartment of some maximal torus contained in $B$ so that $\psi$ reduces to a generic character of $B_u(k_F)$. It is shown in both cases that there exists an admissible embedding $j_0 : S \rw G$ so that the representation $\pi(\chi_S,j_0)$ of $G(F)$ is $(B,\psi)$-generic. Moreover, one has $S_\phi = [\hat S]^\Gamma$, so that $\tx{Irr}(S_\phi) = X^*(\hat S^\Gamma) = X_*(S)_\Gamma=\tb{B}(S)$, where $\tb{B}(S)$ is the set of isomorphism classes of isocrystals with $S$-structure \cite{Ko97}. Using $j_0$ one obtains a map $\tx{Irr}(S_\phi) =\tb{B}(S) \rw \tb{B}(G)_\tx{bas}$. Each $\rho \in \tx{Irr}(S_\phi)$ provides in this way an extended pure inner twist $(G^{b_\rho},b_\rho,\xi_\rho)$. The composition $j_\rho = \xi_\rho\circ j_0$ is an admissible embedding $S \rw G^{b_\rho}$ defined over $F$ and provides by the black box construction alluded to above a representation $\pi(\chi_S,j_\rho)$ of $G^{b_\rho}(F)$. Thus we have
\[ \iota_{B,\psi} : \tx{Irr}(S_\phi) \rw \Pi_\phi,\qquad \rho \mapsto \pi(\chi_S,j_\rho). \]
It is known that for any $j$ the contragredient of $\pi(\chi_S,j)$ is given by $\pi(\chi_S^{-1},j)$. In particular, the contragredient of $\pi(\chi_S,j_0)$ is given by $\pi(\chi_S^{-1},j_0)$. The latter representation is $(B,\psi^{-1})$-generic. Hence, the version of $j_0$ associated to ${^LC}\circ\phi$ and the Whittaker datum $(B,\psi^{-1})$ is equal to $j_0$. We have $S_{^LC\circ\phi}=[\hat S^\Gamma]$, and reviewing the above procedure one sees that the map $\iota_{B,\psi^{-1}}$ corresponding to $^LC\circ\phi$ assigns to each $\rho \in X^*(\hat S^\Gamma)$ the representation $\pi(\chi_S^{-1},j_\rho)$ which is indeed the contragredient of $\pi(\chi_S,j_\rho)$.

Tasho Kaletha\qquad \textsf{tkaletha@math.princeton.edu}\\
Deptartment of Mathematics, Princeton University, Fine Hall, Washington Road, Princeton, NJ 08540

\end{document}